\documentclass[11pt,twoside]{article}

\usepackage{a4}

\usepackage{amssymb,amsmath,amsthm,latexsym}
\usepackage{amsfonts}
  \usepackage{amsfonts}
\usepackage{graphicx}

\usepackage{amsmath, amsfonts}
\usepackage{amssymb, graphicx}
\usepackage{amscd}
\usepackage{textcomp}
\usepackage{palatino}
\usepackage{mathtools}
\usepackage{enumitem}

\newtheorem{theorem}{Theorem}[section]

\newtheorem{conjecture}[theorem]{Conjecture}
\newtheorem{corollary}[theorem] {Corollary}

\newtheorem{proposition}[theorem]{Proposition}
\newtheorem{remark}[theorem]{Remark}

\newcommand{\Q}{{\mathbb Q}}

\newcommand{\N}{{\mathbb N}}

\newcommand{\C}{{\mathbb C}}

\vspace{5cm}

\begin{document}
  
  \label{'ubf'}  
\setcounter{page}{1}                                 

\markboth {\hspace*{-9mm} \centerline{\footnotesize \sc
DERIVATIVES OF $L$-FUNCTIONS }
                 }
                { \centerline                           {\footnotesize \sc  
         T. Chatterjee \& S. Dhillon                                                } \hspace*{-9mm}              
               }

\vspace*{-2cm}

\begin{center}
{ 
       { \textbf { SPECIAL VALUES OF DERIVATIVES OF CERTAIN $L$-FUNCTIONS
                               }
       }
\\

\medskip
{\sc Tapas Chatterjee\footnote{Research of the first author is partly supported by the core research grant CRG/2023/000804 of the Science and Engineering Research
Board of DST, Government of India.} }\\
{\footnotesize  Department of Mathematics,}\\
{\footnotesize Indian Institute of Technology Ropar, Punjab, India.
}\\
{\footnotesize e-mail: {\it tapasc@iitrpr.ac.in}}
\medskip

{\sc Sonika Dhillon\footnote{Research of the second author is partly supported by a post-Doctoral
fellowship at Indian Institute of Technology Delhi.} }\\
{\footnotesize Department of Mathematics, }\\
{\footnotesize Indian Institute of Technology Delhi 
}\\
{\footnotesize e-mail: {\it sonika05@maths.iitd.ac.in}}
\medskip
}
\end{center}

\thispagestyle{empty} 
\vspace{-.4cm}

\hrulefill

\begin{abstract}  
{\footnotesize 
In this paper we address the question of non-vanishing of $L'(0,f)$ where $f$ is an algebraic valued periodic function. In 2011, Gun, Murty and Rath studied the nature of special values of the derivatives of even Dirichlet-type functions   and proved that it can be either zero or transcendental. Here for some special cases we characterize the set of functions for which $L'(0,f)$ is zero or transcendental. Using a theorem of Ramachandra about multiplicative independence of cyclotomic units we also provide some non-trivial examples of functions where $L'(0,f)$ is zero. Finally, assuming Schanuel's conjecture we derive the algebraic independence  of special values of derivatives of $L$-functions.

}
 \end{abstract}
 \hrulefill
 


\vspace{-.37cm}

\section{Introduction}
For any periodic arithmetic function $f,$ the  Dirichlet series associated to the function $f$ is defined as 
 \begin{align*}
 L(s,f)=\sum_{n=1}^\infty \frac{f(n)}{n^s}
 \end{align*}
 for $\Re(s)>1$. 
 Suppose $f$ is a periodic function with period $q,$ then the above Dirichlet series can be re-written as
  \begin{align}\label{8eq5}
  L(s,f)=\sum_{n=1}^\infty \frac{f(n)}{n^s}= q^{-s}\sum_{a=1}^qf(a) \zeta\left(s, \frac{a}{q}\right),
  \end{align}
where $\zeta(s,a)$ is the Hurwitz zeta function  defined as 
\begin{align*}
\zeta(s,a)=\sum_{n=0}^\infty \frac{1}{(n+a)^s}
\end{align*}
for $0<a \leq 1$ and $s \in \C$ with $\Re{(s)}>1$.  Since  the Hurwitz zeta function admits an analytic
 continuation to the whole complex plane with a simple pole at $s=1$ and residue 1, one can extend the 
 Dirichlet series to all values of $s \in \C$ possibly at $s=1$ also provided $\sum_{a=1}^\infty f(a)=0$. 
 
 The study of the special values of the $L$-functions has been the focus for a very long  time and a 
 series of interesting results are known in this direction. In fact the motivation for studying certain number
  theoretic constants such as digamma function at rational arguments, generalized Stieltjes constants comes
   from the desire to understand the arithmetic nature of special values of $L$-functions and the derivatives of $L$-functions. 
 For the case when $s=1$ in Eq. \eqref{8eq5} a lot of development has been done about the arithmetic
  nature of the special values of $L$-functions attached to  a variety of arithmetic functions by many authors.
 In fact  it was shown \cite{DHL} that for
a $q$ periodic arithmetic function $f$ satisfying $\sum_{n=1}^q f(n)=0,$
the Dirichlet series
$L(s, f)$ converges at $s = 1$ and is given by the following closed form expression
 \begin{align*}
 L(1,f)= \sum_{a=1}^qf(a) \gamma(a,q)
\end{align*}
 where the constants $\gamma(a,q)$ are defined as
 \begin{align*}
\gamma(a,q)=\lim\limits_{x \to \infty}\left(\sum\limits_{\substack{0<n\leq x \\ n \equiv a \text{ mod } q}}\frac{1}{n}-\frac{1}{q}\log x\right)
\end{align*}
for $ 1 \leq a \leq q.$
 The work related to the arithmetic nature and linear independence of these
 constants $\gamma(a,q)$ has been done by Murty and Saradha \cite{RM}.

In this article, we investigate the special values of the derivatives of $L$-functions mainly at the points $s=0$.
 Now we will see an interesting relation between the derivatives of $L$-functions at $s=0$ and the logarithms of gamma values with the help
 of the following identities due to Lerch \cite{ML}:
 $$ \zeta(0,a/q) =\frac{1}{2}-\frac{a}{q} \ \ \ \text{ and } \ \ \  \zeta'(0,a/q)= \log \Gamma(a/q)- \frac{1}{2} \log 2 \pi.$$ 
  Taking the derivative on both sides of Eq. \eqref{8eq5}, we get
  \begin{equation}\label{8eq6}L'(s,f)=-\frac{\log q}{q^{s}} \sum_{a=1}^q f(a) \zeta(s, a/q) + q^{-s} \sum_{a=1}^q f(a) \zeta'(s, a/q).
  \end{equation}
 At $s=0,$ using the Lerch identities (see \cite{ML}),
    Eq. \eqref{8eq6}   reduces to
    \begin{align}\label{8eq7}
  L'(0,f)=-\log q \sum_{a=1}^q f(a) (1/2-a/q)+ \sum_{a=1}^q f(a)  \log \left(\Gamma (a/q)\right)- \frac{1}{2} \log (2 \pi) \sum_{a=1}^q f(a).
  \end{align}
  Assume $f$ to be an even Dirichlet-type periodic  function, then Eq. \eqref{8eq7} can be written as
  \begin{align}\label{8eq8}
   L'(0,f)= \sum_{\substack{a=1,\\ (a, q )=1  }}^{q/2} f(a)
     \left(\log \left(\Gamma (a/q)+\log(1-a/q\right)\right)- \log (2 \pi)\sum_{\substack{a=1,\\ (a, q )=1  }}^{q/2} f(a).
  \end{align}
Now by  using the following  reflection formula  for gamma function
\begin{align}\label{8ss1}
\Gamma(z)\Gamma(1-z)=\frac{\pi}{\sin (\pi z)}
\end{align}
  Eq. \eqref{8eq8} reduces to
  \begin{align}\label{8eq9}
  L'(0,f)= -\sum_{\substack{a=1,\\ (a, q )=1  }}^{q/2} f(a)  \log 2 \sin \left(\frac{a\pi}{q}\right).
  \end{align}

  Thus, the linear independence of logarithm of sine function at rational arguments is very closely related to
    the non-vanishing of special values of derivatives of $L$-functions. 
  In 2011, Gun, Murty and Rath \cite{GRP3} studied the non-vanishing of special values of derivatives of $L$-functions and proved the following theorem:
  \begin{theorem}(Gun, Murty and Rath).
  If $f$ is an even Dirichlet-type periodic function that takes algebraic values, then $L'(0, f)$ is either zero or transcendental.
  \end{theorem}
Note that the above theorem  due to Gun, Murty and Rath also holds for the case when $f$ is an even periodic function that assumes algebraic values.
  Now, we will characterize the set of functions where $L'(0,f)$ is zero or transcendental and will prove the following result:
  \begin{theorem}\label{THM1}
  Let $q\neq p^n$ be any positive integer and $q$ satisfies one of the conditions given in Proposition \ref{prop3} (section 2). Let $f$ be any even Dirichlet type algebraic valued function with period $q$. Then $L'(0,f)$ is zero if and only if 
  
  $$f(a)= \left\{ \begin{array}{rcl}
\hspace{-.3cm} c  & \mbox{ \ $(a,q) =1,$ \ \ \ $1 \leq a \leq q/2$}
 \\  \\
0& \mbox{otherwise}
\end{array}\right.$$
for some algebraic number $c$.
  \end{theorem}
  
  \begin{corollary}
  Let $q$ be any positive integer as defined in Theorem \ref{THM1}. Then $L'(0,f)$ is transcendental if and only if $f$ is a non constant Dirichlet type even function.
  \end{corollary}
  
 The above theorem deals with the case when $q$ is not a prime power and satisfy some interesting properties. Next we will consider the case of prime  powers and here is the statement of the theorem:
  \begin{theorem}\label{THM2}
  Let $q=p^n$ be any positive integer for some prime $p$. Let $f$ be any even Dirichlet type algebraically valued periodic function with period $q$. Then $L'(0,f)$ is transcendental if and only if $f$ is not an identically zero function.
  \end{theorem}

Note that the above theorems implies that $L'(0,f)$ can vanish for infinitely many arithmetic functions and also
 can take transcendental values for infinitely many functions. Thus, as a corollary to  Theorem \ref{THM1} and \ref{THM2}, we have the following result:
\begin{corollary}
Let $q>1$ be an integer. Let $f$ be an even algebraic valued Dirichlet-type periodic function with period $q$. 
Then arithmetic nature of $L'(0,f)$ changes infinitely often for infinitely many functions.
\end{corollary}

So far we have considered the case of prime powers and product of primes satisfying certain conditions. Now for the remaining cases when $q$ is a product of at least two prime factors and $q$ does not satisfy any of the conditions given in Proposition \ref{prop3}, we have the following remark:

 \begin{remark}
For the case when $q$ does not satisfy any of the conditions in Proposition \ref{prop3} and  $f$
 is not a constant function, we will use the following theorem 
of Ramachandra (see Theorem 2 of  \cite{KR}) to find an explicit example where $L'(0,f)$ is zero for a non-trivial Dirichlet type algebraic valued function :
\begin{theorem}\label{RAMA1}(Ramachandra).
Let $q>1$ be a positive integer with at least two odd prime divisors  $p_1,p_2$.
Assume $p_1$ have the property that the residue class group modulo $p_1$ has a non-principal character
$\chi$ with $\chi(-1)=1$ and $p_2 \equiv 1 \text{ mod } p_1.$ Then the units 
\begin{align*}
\frac{a^s-1}{a-1}, \ \  2 \leq s \leq q/2,\ \ (s,q)=1
\end{align*}
where $a$ is any primitive root modulo $q,$ are multiplicatively dependent.
\end{theorem}

\textbf{Example:} Suppose $q=155,$  then $o(5)=3(\text{ mod }31)$ and hence $q$ does not satisfy any of the conditions in 
Proposition \ref{prop3}. Define a non-principal real character $\chi_0$ modulo 5 such that
\begin{align*}
\chi_0(1)=\chi_0(4)=1 \ \text{ and }  \  \chi_0(2)=\chi_0(3)=-1.
\end{align*}
Clearly $q$ satisfy the conditions given in Theorem \ref{RAMA1} and $\chi_0$ is an even character.
Define $\mathfrak{R_q}$ be the multiplicative group of residue classes prime to $q$ modulo the subgroup generated by classes
 $1$ and $-1$.  Then order of the group $\mathfrak{R_q}$ is 60. Extend $\chi_0$ to a character $\chi$ of $\mathfrak{R_q},$
  then by using the Theorem \ref{RAMA1}, we have 
\begin{align*}
\bigg|\prod \limits_{\substack{(s,q)=1 \\ 2 \leq s \leq q/2}}\left(\frac{a^s-1}{a-1}\right)^{\chi(s)}\bigg|=1.
\end{align*}
Now taking $a=\zeta_q$ which is a primitive $q$-th root of unity, we get
\begin{align}\label{8eq10}
\bigg|\prod \limits_{\substack{(s,q)=1 \\ 2 \leq s \leq q/2}}\left(\frac{\zeta_q^s-1}{\zeta_q-1}\right)^{\chi(s)}\bigg|=1
\end{align}
We also need the following identity of the sine function at rational arguments (see Lemma 3.3 of \cite{CD1})
\begin{equation}\label{8eq11}
2^{\phi(q)}\prod_{\substack{k=1,\\ (k, q )=1  }}^{q-1}  \sin\left( \frac{ k\pi}{q}\right)=1,
\end{equation}
when $q$ is not a prime power.
Since $|1-\zeta_q^k|=2 \sin \left(\frac{k \pi}{q}\right),$ substituting the value of $\log 2 \sin (\pi/q)$ from
 Eq. \eqref{8eq11} in Eq. \eqref{8eq10} we have
\begin{align}\label{8eq12}
\bigg|\prod \limits_{\substack{(s,q)=1 \\ 2 \leq s \leq q/2}}\left(1-\zeta_q^s\right)^{\chi(s)
+\sum \limits_{\substack{(b,q)=1 \\ 2 \leq b \leq q/2}}\chi(b)}\bigg|=1
\end{align}
and for the character $\chi,$ one can easily calculate that the sum
\begin{align*}
\sum \limits_{\substack{(b,q)=1 \\ 2 \leq b \leq q/2}}\chi(b)=-1.
\end{align*}
Thus Eq. \eqref{8eq12} reduces to
\begin{align}\label{8eq13}
\sum \limits_{\substack{(s,q)=1 \\ 2 \leq s \leq q/2}}\left(\chi(s)-1\right)\log 2 \sin \left(\frac{s\pi}{q}\right)=0
\end{align}
and hence the numbers $ \log2 \sin (s\pi/q)$ with $(s,q)=1$ and $2 \leq s \leq q/2$ are linearly dependent 
over the field of algebraic numbers 
by using Baker's theorem. Now define an even Dirichlet-type periodic function with period $q$ such that 
$f(s)-f(1)= \chi(s)-1$ where $2 \leq s \leq q/2$ 
with $(s,q)=1.$ Then by using Eq. \eqref{8eq9}, \eqref{8eq11} and \eqref{8eq13}, we have
\begin{align*}
 L'(0,f)= -\sum\limits_{\substack{s=2 \\ (s,q)=1}}^{q/2}\bigg( f(s)-f(1) \bigg) \left(\log 2\sin \frac{\pi s}{q}  \right)= 
  -\sum\limits_{\substack{s=2 \\ (s,q)=1}}^{q/2} \left(\chi(s)-1\right)  \left(\log 2\sin \frac{\pi s}{q}  \right)=0,
\end{align*}
and $\chi$ is clearly a non-constant function. Thus for different values of $f(1)$ there exists infinitely many 
non-constant functions $f$ such that $L'(0,f)=0$.
\end{remark}

In fact by using a similar example for the case when $q=55$, we have constructed a counter example to the variant of Rohrlich-Lang conjecture about the linear independence of logarithm of gamma function due to Gun, Murty and Rath (see \cite{CD21} and \cite{CD6}).

\section{\bf Notations and Preliminaries}
 
 In this section, we list some of the known results  that will help us to develop a setting for our results.
 The most important ingredient for many of our results is the following theorem of Baker about the 
 linear forms in logarithms of algebraic
numbers (see \cite{AB} ).

\begin{proposition}\label{prop1} If $\alpha_1,...,\alpha_n$ are non-zero algebraic numbers such that $\log  \alpha_1,...,\log  \alpha_n$
are linearly independent over the field of rational numbers, then $1, \log  \alpha_1,...,\log \alpha_n$ are linearly 
independent over the field of algebraic numbers.
\end{proposition}

The following proposition by Pei  and Feng (\cite{DY}) is of great advantage in finding the  necessary and sufficient conditions for 
multiplicative independence cyclotomic units for prime powers and product of distinct primes  (for a proof  see \cite{DY}). Before 
that we define $n$  a semi-primitive root mod $q,$ if the order of $n($mod $q)$  is  $\frac{\phi(q)}{2} $.
\begin{proposition}\label{prop3}
For a composite number $m$ with $m \not \equiv 2 ( \text{ mod } 4),$ the system 
$$\bigg \{\frac{1-\zeta_m^h}{1-\zeta_m}\big | \ (h,m)=1, \ 2 \leq h < m/2\bigg\}$$
of cyclotomic units of field $\Q(\zeta_m)$  is independent if and only if one of the following conditions 
are satisfied (here $\alpha_0 \geq 3; \alpha_1, \alpha_2, \alpha_3 \geq 1; p_1, p_2 ,p_3$ are
odd primes): \\
 (I) $m=4p_1^{\alpha_1}$; and 
 
(I, 1) 2 is a primitive root mod $p_1^{\alpha_1}$; or 

(I, 2) 2 is a semi-primitive root mod $p_1^{\alpha_1}$ and $p_1 \equiv 3 \text{ (mod } 4)$.\\
(II) $m=2^{\alpha_0}p_1^{\alpha_1}$; the order of $p_1( \text{mod } 2^{\alpha_0})$ is $2^{\alpha_0-2}$, 
$2^{\alpha_0-3}p_1 \not \equiv-1( \text{ mod }2^{\alpha_0}),$ and

(II, 1) 2 is a primitive root mod $p_1^{\alpha_1}$; or 

(II, 2) 2 is a semi-primitive root mod $p_1^{\alpha_1}$ and $p_1 \equiv 3 \text{ (mod } 4)$.\\
(III) $m=p_1^{\alpha_1}p_2^{\alpha_2}$; and 

(III,1) when $p_1 \equiv p_2 \equiv 3 \text{ (mod }4)$: $p_1$ is a semi-primitive root mod 
$p_2^{\alpha_2}$ and $p_2$ is a semi-primitive root mod $p_1^{\alpha_1},$ or vice versa.

(III, 2) otherwise: $p_1$ and $p_2$ are primitive root mod $p_2^{\alpha_2}$ and mod $p_1^{\alpha_1}$ respectively. \\
(IV) $m=4p_1^{\alpha_1}p_2^{\alpha_2}; (p_1-1,p_2-1)=2$ and 

(IV, 1) when $p_1 \equiv p_2 \equiv 3 \text{ (mod }4)$: 2 is a primitive root for one $p$ and a semi-primitive root 

for another $p$; $p_1$ is primitive root mod $2p_2^{\alpha_2}$ and $p_2$ is a semi-primitive root mod $2p_1^{\alpha_1}$ or

 vice versa.

(IV, 2)  when $p_1 \equiv 1, p_2 \equiv 3 \text{ (mod } 4)$: 2 is a primitive root mod $p_2^{\alpha_2}$; $p_1$ and $p_2$ are primitive

 root mod $p_2^{\alpha_2}$ and mod $p_1^{\alpha_1},$ respectively.\\
(V)$m=p_1^{\alpha_1}p_2^{\alpha_2}p_3^{\alpha_3}$; $p_1 \equiv  p_2 \equiv p_3 \equiv 
3 \text{ (mod } 4)$: $(p^i-1)/2 \ \  ( 1 \leq i \leq 3)$ are co-prime to each other; and 

(V,1) $p_1, p_2, p_3$ are primitive root mod $p_2^{\alpha_2},$ mod $p_3^{\alpha_3},$ mod $p_1^{\alpha_1},$ respectively and semi-primitive
 root mod $p_3^{\alpha_3}$, mod $p_1^{\alpha_1}$, mod $p_2^{\alpha_2}$, respectively.
\end{proposition}

The next proposition is of great use as it extends the idea of multiplicative independence of cyclotomic units in 
Proposition \ref{prop3} for the case when $q \equiv 2 \text{ (mod 4)}$ (for a proof see \cite{CD2}).

\begin{proposition}\label{prop4}Let $q \geq 2$ be a positive integer. Then the numbers
$$ \left \{ \log \left( 2 \sin \frac{a\pi}{q} \right) : 1 < a < \frac{q}{2} , ~(a,q)=1,\ \ a/q \neq  \frac{1}{\pi}\left( \sin^{-1} \frac{1}{2^\alpha} \right),
\ \alpha \in \Q   \right \}, \  \ \pi  \text{ \ \ and  \ \ } \log 2,  $$ are linearly independent over the field of algebraic numbers  if and only if
$q$ satisfy one of the following conditions:\\
a) $q$ is a prime power \\ b) $q=2p^n,$ where $p$ is an odd prime and $n \in \N$\\
 c) $q$ satisfies the conditions  in proposition 2.3.\\
 d) $q=2m$ where $m$ satisfies conditions III and V in proposition 2.3.

\end{proposition}

Another result that we require to prove the non-vanishing of derivatives of $L$-functions for the case of primes and their powers is the following \cite{CD2}:
\begin{proposition}\label{prop5} Let $q \geq 2$ be an integer. Then the numbers
$$ \left \{ \log \left( 2 \sin \frac{a\pi}{q} \right) : 1 \leq a < \frac{q}{2} , ~(a,q)=1,\ \ a/q \neq  \frac{1}{\pi}\left( \sin^{-1} \frac{1}{2^\alpha} \right),
 \alpha \in \Q   \right \},   \ \pi  \text{ \ \ and  \ } \log 2,  $$ are linearly independent 
over the field of algebraic numbers if and only if $q$ is a prime power or $q=6.$
\end{proposition}

\section{Proof of the main theorems}

 \subsection{Proof of Theorem \ref{THM1} }
  \begin{proof}
  
  Observe that by using Eq. \eqref{8eq7}, for any Dirichlet-type periodic arithmetic function $f,$ we have
 \begin{align}\label{8eq17}
 L'(0,f)=-\log q \sum\limits_{\substack{a=1 \\ (a,q)=1}}^{q} f(a) (1/2-a/q)+
   \sum\limits_{\substack{a=1 \\ (a,q)=1}}^{q} f(a)  \log \left( \Gamma (a/q)\right)- \frac{1}{2} \log (2 \pi) \sum\limits_{\substack{a=1 \\ (a,q)=1}}^{q} f(a).
   \end{align}
  Since $f$ is an even Dirichlet-type periodic function, therefore Eq. \eqref{8eq17} can be re-written as 
  $$L'(0,f)= \sum\limits_{\substack{a=1 \\ (a,q)=1}}^{[q/2]} f(a)\big(\log \left(\Gamma (a/q)\right) 
  + \log \left(\Gamma(1-a/q)\right) \big)-  \log (2 \pi)\sum\limits_{\substack{a=1 \\ (a,q)=1}}^{[q/2]}f(a) .$$
  Now by using the reflection formula $\Gamma(z)\Gamma(1-z)=\pi/ \sin(\pi z)$ for gamma function, the above equation reduces to
   \begin{align}\label{8eq18}L'(0,f)= \sum\limits_{\substack{a=1 \\ (a,q)=1}}^{[q/2]}f(a)\left( \log \pi-\log \sin \frac{a\pi}{q} \right)-
    \log (2 \pi) \sum\limits_{\substack{a=1 \\ (a,q)=1}}^{[q/2]} f(a).
    \end{align}
 Also by  performing some simple calculations and re writing Eq. \eqref{8eq18}, we get
   \begin{align}\label{8eq19}
   L'(0,f)= -\sum\limits_{\substack{a=1 \\ (a,q)=1}}^{[q/2]}f(a) \left(\log 2\sin \frac{\pi a}{q}  \right).
   \end{align}

   Since $q$  is a product of at least two prime factors, then the sine function satisfies the 
  following identity (for a proof see Lemma 3.3 of \cite{CD1}):
   \begin{equation}\label{8eq20}
2^{\phi(q)}\prod_{\substack{k=1,\\ (k, q )=1  }}^{q-1}  \sin\left( \frac{ k\pi}{q}\right)=1
\end{equation}
    By using the above sine identity, Eq. \eqref{8eq19} can be written as
 \begin{align}\label{8eq21}
  L'(0,f)= \sum\limits_{\substack{a=2 \\ (a,q)=1}}^{q/2}\bigg( f(a)-f(1) \bigg) \left(\log 2\sin \left(\frac{a\pi }{q}\right)  \right).
  \end{align}
   
   Since $q$ satisfies one of the conditions given in Proposition \ref{prop3}, therefore by using Proposition 
\ref{prop4} the numbers $\log 2 \sin (a\pi/q)$ where $(a,q)=1$ with $2 \leq a \leq q/2$ are 
linearly independent over the field of algebraic numbers. Thus, $L'(0,f)=0$ if and only if $f(a)=f(1)$ for 
all $a$ and hence $f$ is a constant function. This completes the proof.

  \end{proof}

\subsection{Proof of Theorem \ref{THM2} }
\begin{proof}
Suppose $f$ is an even Dirichlet type algebraically valued function, then by using Theorem \ref{THM1}, we have
\begin{align}\label{8EQ19}
   L'(0,f)= -\sum\limits_{\substack{a=1 \\ (a,q)=1}}^{[q/2]}f(a) \left(\log 2\sin \frac{\pi a}{q}  \right).
   \end{align}
   Since $q$ is a prime power,  therefore by using Proposition \ref{prop5}, 
    the numbers $\log(2 \sin (\pi a/q))$ where $ 1 \leq a \leq q/2$ with $(a,q)=1$
   are linearly independent over $\overline{\Q}.$ Hence by using Eq. \eqref{8EQ19}, $L'(0,f)=0$ if and only if $f$ is  an identically zero function. This completes the proof.
   
\end{proof}
Observe that the above discussion leads us to the following remark about the non-vanishing of the special values of the derivatives of $L$-functions for the case when $q$ does not satisfy any of the conditions in Proposition \ref{prop3}
.\begin{remark}
For each $q$ which does not satisfy any of the conditions given in Proposition \ref{prop3}, there exists infinitely many even Dirichlet type non-constant functions such that $L'(0,f)$ is zero.
\end{remark}

\section{Consequences of weak Schanuel's conjecture}
In this section, we discuss  the algebraic independence of the special values of derivatives of $L$-functions under the assumption of the weak Schanuel's conjecture. We begin with the following result for the linear independence of the special values of derivatives of $L$-functions. Here is the statement of the theorem:
\begin{theorem}
Let $q>1$ be any prime power and $f_1, \cdots, f_n$ be even Dirichlet type functions with period $q$. Then $L'(0,f_1), \cdots, L'(0,f_n)$ are linearly independent over the field of algebraic numbers if and only if $f_1, \cdots, f_n$ are linearly independent functions over the field of algebraic numbers.
\end{theorem}
\begin{proof}
If possible, let there exists algebraic numbers $c_i^{~,}s$ not all zero such that
\begin{align*}
\sum\limits_{\substack{i=1}}^{n}c_iL'(0,f_i)=0. 
\end{align*}
Substituting the value of $L'(0,f_i)$ in the above equation, we have
\begin{align*}
\sum\limits_{\substack{i=1}}^{n}c_i\left(\sum\limits_{\substack{a=1 \\ (a,q)=1}}^{[q/2]}f_i(a)\log 2 \sin \left(\frac{a\pi}{q}\right)\right)=0
\end{align*}
and re-writing above equation we get
\begin{align*}
\sum\limits_{\substack{a=1 \\ (a,q)=1}}^{[q/2]}\left(\sum_{i=1}^nc_if_i(a)\right)\log 2 \sin \left(\frac{a\pi}{q}\right)=0
\end{align*}
and thus by using Proposition \ref{prop5}, we have
\begin{align*}
\sum_{i=1}^n c_if_i(a)=0
\end{align*}
for all $a$. This completes the proof.

\end{proof}

Now in our next theorem, we consider the algebraic independence of the special values of the derivatives of the $L$-functions under the assumption of weak Schanuel's conjecture:
\begin{conjecture}
Let $\alpha_1, \cdots, \alpha_n$ be non-zero algebraic numbers such that the numbers $\log \alpha_1, \cdots, \log \alpha_n$ are $\Q$-linearly independent. Then $\log \alpha_1, \cdots, \log \alpha_n$ are algebraically independent.
\end{conjecture}
Under the assumption of the above conjecture, we prove the following theorem:

\begin{theorem}
Let $q>1$ be any prime power and $f_1, \cdots, f_n$ be even Dirichlet type functions with period $q$. Suppose $f_1, \cdots, f_n$ are algebraically independent functions, then weak Schanuel's conjecture implies $L'(0,f_1), \cdots, L'(0,f_n)$ are algebraically independent. 
\end{theorem}

\begin{proof}
Suppose there exists a  polynomial $P(x_1, \cdots, x_n)$ not identically zero such that
\begin{align*}
P(L'(0,f_1), \cdots, L'(0,f_n))=0.
\end{align*}
Substituting the value of $L'(0,f_i)$ in the above equation, we get an equation of the form
\begin{align*}
Q(\log 2 \sin \pi /q, \cdots, \log 2 \sin a \pi/q )=0
\end{align*}
for some polynomial $Q(x_1,\cdots, x_n)$ where $2 \leq a \leq q/2$ with $(a,q)=1.$ Also observe that the numbers $\log 2 \sin a \pi/q$ are linearly independent by using Proposition \ref{prop5}, then by using weak Schanuel's conjecture we must have $Q(x_1, \cdots, x_n) \equiv 0$. Thus by using,
\begin{align*}
L'(0,f)= -\sum\limits_{\substack{a=1 \\ (a,q)=1}}^{[q/2]}f(a) \left(\log 2\sin \frac{\pi a}{q}  \right).
\end{align*}
and the fact that $Q(x_1, \cdots, x_n)\equiv 0$ gives us a non-trivial algebraic relation between the functions $f_1, \cdots, f_n$ which is a contradiction.
\end{proof}

\section{Declarations}

\noindent
{\bf Competing Interests:} The authors have no competing interests to declare that are relevant to the content of this article.\\

\noindent
{\bf Data availability:} Data sharing is not applicable to this article as no datasets were generated
or analyzed during the current study.


\begin{thebibliography}{10}


\bibitem{AB}
A. Baker, {\em Transcendental Number Theory}, Cambridge Univ. Press,  (1975).


\bibitem{CD0}
T. Chatterjee and S. Dhillon, {\em Linear Independence of Harmonic Numbers over the field of Algebraic Numbers}, the Ramanujan Journal, Springer, 51 (2020), no. 1, 43-66.



\bibitem{CD1}
T. Chatterjee and S. Dhillon, {\em Linear Independence of Harmonic Numbers over the field of Algebraic Numbers II}, the 
Ramanujan Journal, Springer, 52 (2020) 555-580.

\bibitem{CD2}
T. Chatterjee and S. Dhillon, {\em Linear independence of logarithms of cyclotomic numbers and a conjecture 
of Livingston}, Canad. Math. Bull., 63 (2020) 31-45.


\bibitem{CD21}
T. Chatterjee and S. Dhillon, {\em A note on a variant of conjecture of Rohrlich}, Mathematika, 68 (2) (2022) 400-415.

\bibitem{CD22}
T. Chatterjee and S. Dhillon, {\em On a conjecture of Murty-Saradha about digamma values}, Monatsh. Math. 199 (2022), no. 1, 23-43.

\bibitem{CD6}
S. Gun, M. Ram Murty and P. Rath, {\em Linear independence of digamma function and a
variant of a conjecture of Rohrlich }, Journal of Number Theory 129 (8) (2009) 1858-1873.









\bibitem{GRP3}
S. Gun, M. R. Murty and P. Rath, {\em Transcendental nature of special values of $L$-function }, Canad. J. Math. 63(1) (2011) 136-152.






\bibitem{DHL}
D. H. Lehmer, {\em Euler constants for arithmetical progressions}, Acta Arithmetica 27 (1975)
125–142.

\bibitem{ML}
M. Lerch, {\em  Dalsi studie v oboru Malmstenovskych rad}, Rozpravy Ceske Akad.  18 (3) (1894) 63pp.






\bibitem{RM}
M. Ram  Murty and N. Saradha, {\em  Euler-Lehmer constants and a conjecture of Erd\H{o}s}, J. Number Theory 130 (12)  (2010) 2671-
2682.





\bibitem{DY}
D. Y. Pei and  K. Feng, {\em  On independence of the cyclotomic units}, Acta Math. Sinica 23 (1980)
773-778.

\bibitem{KR}
K. Ramachandra, {\em On the units of cyclotomic fields},  Acta Arithmetica 12 (1966) 165-173.



\end{thebibliography}
\end{document}